\documentclass{amsart}
\usepackage{amssymb,amsmath,latexsym}      
\usepackage{times}

\newcommand{\showhide}[1]{#1} 
\showhide{\usepackage{tikz}}

\numberwithin{equation}{section}

\newtheorem{thm}{Theorem}[section]

\newtheorem{cor}[thm]{Corollary}
\newtheorem{conj}[thm]{Conjecture}
\newtheorem{prop}[thm]{Proposition}
\newtheorem{lemma}[thm]{Lemma}

\newcommand{\cl}{\mathrm{cl}}

\title{Inconsquential results on the Merino-Welsh conjecture for Tutte polynomials}
\author{Joseph P.S. Kung}

\begin{document}

\begin{abstract}    The Merino-Welsh conjectures say that subject to conditions, there is an inequality among the Tutte-polynomial evaluations $T(M;2,0)$, $T(M;0,2)$, and $T(M;1,1)$. We present three results on a Merino-Welsh conjecture.  These results are ``inconsequential'' in the sense that although they imply a version of the  conjecture for many matroids,  they seem to be dead ends.   
\end{abstract}

\maketitle



\section{Merino-Welsh inequalities} \label{Intro}
There are many variations on the Merino-Welsh conjecture.  They are of the form ``subject to conditions, there is an inequality involving the evaluations
$T(M;2,0)$, $T(M;0,2)$, and $T(M;1,1)$ of the Tutte polynomial of the matroid $M$''.  We shall consider the following version. 

\begin{conj}\label{MW} 
If a matroid $M$ has no loops or isthmuses, then 
\[
T(M;2,0)+ T(M;0,2) \geq 2 T(M;1,1),
\eqno(1.1)\]
with equality if and only if $M$ is a direct sum of $U_{1,2}$, the rank-$1$ matroid consisting of two parallel elements.  
\end{conj}
\noindent
The focus of this paper 
is on finding sufficient conditions on a matroid so that inequality (1.1) holds.   

As background, we note that $T(M;2,0)$ is the number of no-broken-circuit or nbc sets and $T(M;1,1)$ the number of bases in $M$.  Hence Conjecture \ref{MW} asserts that the average number of nbc sets in $M$ and its dual exceeds the average number of bases in $M$ and its dual.  A counting or algebraic-geometry proof of Conjecture \ref{MW} is quite possible.  From the counting interpretation, $T(M;2,0)$ and $T(M;0,2)$ are positive integers.  We shall tacitly use this fact, as well as the fact due to Tutte that the coefficients of his polynomial are non-negative.     

\section{A density condition} \label{sect:density}

The first result is that a sufficiently dense matroid satisfies Conjecture \ref{MW}.  
 
\begin{thm} \label{density}
Let $M$ be an isthmus-free matroid having size $n$ and rank $r$. Then if $r \geq 4$ and
\[
n \geq \lceil r(\log r +  \log\log r + \log\log\log r) \rceil,
\]
with the logarithms in base $2$, then $T(M;0,2) \geq 2T(M;1,1)$. In particular, inequality (1.1) holds for $M$.
\end{thm}

The theorem is true for the simple reason that exponential growth beats polynomial growth in the long run.      

\begin{prop}\label{inq1}  If $r \geq 4$ and $n \geq \lceil r(\log r + \log\log r + \log\log\log r) \rceil,$ then $2^{n-r} \geq 2 \binom {n}{r}$. 
\end{prop} 




Consider the function 
\[
f(n,r) = 2^{n-r} - 2 \binom {n}{r} = 2^{n-r} -  \frac {2n(n-1)\cdots (n-r+1)}{r!},
\] 
where, in spite of the notation, $n$ could be a real number, although $r$ is always a positive integer.   
Suppose $n > 2r$.  Then its forward difference in the variable $n$ satisfies   
\begin{eqnarray*}
f(n+1,r) - f(n,r)  &=& [2^{n+1-r} - 2^{n-r}] - 2\left[\binom {n+1}{r} - \binom {n}{r}\right]
\\
&=& 2^{n-r} - 2\binom {n}{r-1}
\\
&> &   2^{n-r} - 2 \binom {n}{r} = f(n,r).
\end{eqnarray*}  
This yields the following lemma.  

\begin{lemma} If $n > 2r$, then 
$
f(n+1,r) > 2f(n,r).
$
In particular, if $f(n,r) \geq 0,$ then $f(n+1,r) > 0$.  
\end{lemma}

By the lemma, to prove Proposition \ref{inq1}, it suffices to find a number $n_r$ such that $n_r \leq \lceil r(\log r +  \log\log r + \log\log\log r) \rceil$ and $2^{n_r-r} \geq 2 \binom {n_r}{r}$.  For small values of $r$, this can be done by a hand calculation.  The results are shown in the following table: 

\begin{center}  
\begin{tabular}{lclclclclclclclclclclclclclclclcl}
\hline 
$r$  \rule{0pt}{15pt} &       $1$ & $2$ & $3$ & $4$  & $5$  & $6$ & $7$ & $8$ & $9$  & $10$  & $11$ & $12$ &$13$ &$14$ & $15$ & $16$
\\
\hline
\hline 
$n_r$  \rule{0pt}{15pt} &  $4$ & $8$ & $12$ & $16$  & $21$  & $25$ & $29$ & $33$ & $37$  & $42$  & $46$ & $50$ & $54$ & $59$& $64$ & $68$ 
\\
\hline
\\
\end{tabular}
\end{center}
This allows us to assume $r \geq 17$. 
\footnote{There is nothing special about $17$; it just happens that $16$ entries will fit comfortably in a one-line table.    Any sufficiently large number, like $8$, will work.} 
We begin with the fact that for $x \geq 17$, 
\[
(\log x)(\log\log x)   \geq \log x+  \log\log x + \log\log\log x + 1.
\]
This is an easy exercise in calculus.   
\footnote{To remind myself, with my memory failing in old age, let $f(x)= (\log x)(\log\log x)$ and $g(x) = \log x + \log\log x+\log\log\log x + 1$.  Then $f(17) \geq g(17) > 0$ and when we differentiate both functions, it is evident that 
\[
\frac {\log\log x}{x} + \frac {1}{x} =  f'(x) \geq g'(x) = \frac {1}{x} +  \frac {1}{x \log x} + \frac {1}{x (\log x) (\log\log x)}> 0
\] 
when $x \geq 17$. The added $1$ in $g(x)$ is there so that we can round up.} 
Thus,
\begin{align*} 
& 2^{n-r} \geq \frac {2^{r(\log r + \log\log r + \log\log\log r)}}{2^r}  
\\
&= \frac{[r(\log r)(\log\log r)]^{r}}{2^r}  \geq \frac {\lceil r(\log r +  \log\log r + \log\log\log r) \rceil^r}{2^r}  = \frac {n^r}{2^r}  \geq 2 \binom {n}{r}.
\end{align*} 
In the last step, we use $2^{r+1} \leq r!$ when $r \geq 5$. This completes the proof of Proposition \ref{inq1}.  

Returning to the proof of Theorem \ref{density}, observe that if a matroid $M$ has no isthmuses, then $T(M;x,y) = y^{n-r} + \cdots,$ where the terms of lower degree in $y$ have non-negative coefficients.  In addition, $T(M;1,1)$ is the number of bases and is bounded above by $\binom {n}{r}$.  Hence, by Proposition \ref{inq1}, under the hypothesis of the theorem,  
\[
T(M;0,2) \geq 2^{n-r} \geq 2\binom {n}{r} \geq 2T(M;1,1).  
\]

We remark that Theorem \ref{density} says that for a fixed rank $r$, there are only finitely many matroids to check to show that Conjecture \ref{MW} holds at that rank.  


\section{Cocircuits}\label{sect:cocircuits}

A \emph{cocircuit} is the complement of a copoint or flat of rank $r-1$.  A matroid is isthmus-free if and only if it has no cocircuits of size $1$; thus, one can regard the condition that there are no isthmuses as a condition on the minimum size of a cocircuit.    The following theorem says that Conjecture \ref{MW} holds if one assumes that all cocircuits are sufficiently large.   

\begin{thm} \label{cocircuits}
Suppose that every cocircuit in a rank-$r$ size-$n$ matroid $M$ has at least $r + 1$ elements.    Then  $M$ satisfies inequality (1.1).
\end{thm}

We begin with an easy binomial identity.  

\begin{lemma} 
\begin{eqnarray*} 
\sum_{j=0}^{r-1}   \binom {r-j-1}{j}  y^{r-1-j}  &=& \binom {2r-2}{r-1} + \binom {2r-3}{r-2} y + \cdots + \binom {r}{1} y^{r-2} + y^{r-1} 
\\
&=& 2^{2r-2}.  
\end{eqnarray*}
\end{lemma}
\noindent 
For example, 
\[
\binom {6}{3}2^0 +  \binom {5}{2}2^1 + \binom {4}{1}2^2 + \binom {3}{0}2^3 = 20 + 10\cdot 2 + 4 \cdot 2^2 + 2^3 = 64 = 2^6.
\]

Now observe that if every cocircuit in the matroid $M$ has size at least $m$, then 
\[
T(M;x,y) = y^{n-r} + \binom {r}{1}y^{n-r-1} + \cdots +\binom {r+m-3}{m-2}y^{n-r-m+2} + p(x,y), 
\]
where $p(x,y)$ is a polynomial with non-negative coefficients of degree $n-r-m+1$ in $y$.  In particular, if $m = r+1,$ then 
\[
T(M;0,2) \geq 2^{n-2r+1}\left( \binom {2r-2}{r-1} + \cdots + \binom {r}{1} 2^{r-2} + 2^{r-1} \right) = 2^{n-1}.  
\]
The hypothesis on cocircuits implies that $n \geq 2r$ and under this condition, $2^{n-1} \geq 2 \binom {n}{r}$.  We conclude that inequality (1.1) holds.  

Theorem \ref{cocircuits} is not the sharpest possible.  It is possible to replace the upper bound of $r+1$ by a bound of the form $r -c\sqrt{r}$ (with $c$ a constant)   for sufficiently large $r$ using the normal approximation to the binomial distribution.    It is also possible to combine the arguments to obtain sufficient conditions involving density and size of cocircuits for inequality (1.1) to hold, but neither refinement would prove the full conjecture.  
We end with one further result of this type.  

\begin{prop} 
Let $M$ be a rank-$r$ size-$n$ isthmus-free matroid with at least $r-1$ loops.  Then inequality (1.1) holds for $M$.
\end{prop}
\begin{proof} 
Loops are never in a basis.  Hence, 
\[
T(M;0,2) \geq 2^{n-r} \geq 2\binom {n-r+1}{r} \geq 2T(M;1,1)
\]
and inequality (1.1) follows.  
\end{proof}

\end{document}